\documentclass[12pt, a4paper, leqno]{amsart}
\usepackage[utf8]{inputenc}
\usepackage{amsmath}
\usepackage{amsfonts}
\usepackage{amssymb}
\usepackage{times}
\usepackage[T1]{fontenc} %skandit 
\usepackage{url} 
\usepackage{color,esint}
\usepackage[dvipsnames]{xcolor}
\usepackage[colorlinks, allcolors=RedViolet,pdfstartview=,pdfpagemode=UseNone]{hyperref} 
\usepackage{graphicx} 
\usepackage{enumitem}
\usepackage{tabularx}
\usepackage{mathtools}
\usepackage{pgf,tikz}
\usepackage{mathrsfs}
\usetikzlibrary{arrows}
\usepackage{array,multirow}
\usepackage{booktabs}

\let\oldmarginpar\marginpar
\renewcommand\marginpar[1]{\-\oldmarginpar[\raggedleft\footnotesize #1]%
{\raggedright\footnotesize #1}}

\usepackage{amsthm}
\theoremstyle{plain}
\newtheorem{thm}[equation]{Theorem}
\newtheorem{lem}[equation]{Lemma}
\newtheorem{prop}[equation]{Proposition}
\newtheorem{cor}[equation]{Corollary}

\theoremstyle{definition}
\newtheorem{defn}[equation]{Definition}
\newtheorem{assumptions}[equation]{Assumption}

\newtheorem{eg}[equation]{Example}

\theoremstyle{remark}

\numberwithin{equation}{section}

\newcommand{\R}{\mathbb{R}}
\newcommand{\N}{\mathbb{N}}
\newcommand{\Rn}{{\mathbb{R}^n}}

\renewcommand{\phi}{\varphi}
\def\le{\leqslant}
\def\leq{\leqslant}
\def\ge{\geqslant}
\def\geq{\geqslant}
\def\phi{\varphi}
\def\rho{\varrho}
\def\vartheta{\theta}

\newcommand{\Phiw}{\Phi_{\text{\rm w}}}
\newcommand{\Phic}{\Phi_{\text{\rm c}}}

\def\esssup{\operatornamewithlimits{ess\,sup}}

\def\diam{\qopname\relax o{diam}}
\def\dist{\qopname\relax o{dist}}

\DeclareMathOperator{\divop}{div}
\renewcommand{\div}{\divop}

\def\loc{{\rm loc}}

\newcommand{\inc}[1]{\hyperref[def:aInc]{{\normalfont(Inc){\ensuremath{_{#1}}}}}}
\newcommand{\dec}[1]{\hyperref[def:aDec]{{\normalfont(Dec){\ensuremath{_{#1}}}}}}
\newcommand{\ainc}[1]{\hyperref[def:aInc]{{\normalfont(aInc){\ensuremath{_{#1}}}}}}
\newcommand{\ainci}[1]{\hyperref[def:aInc]{{\normalfont(aInc){\ensuremath{_{#1}^\infty}}}}}
\newcommand{\adec}[1]{\hyperref[def:aDec]{{\normalfont(aDec){\ensuremath{_{#1}}}}}}
\newcommand{\adeci}[1]{\hyperref[def:aDec]{{\normalfont(aDec){\ensuremath{_{#1}^\infty}}}}}
\newcommand{\azero}{\hyperref[def:a0]{{\normalfont(A0)}}}
\newcommand{\aone}{\hyperref[def:a1]{{\normalfont(A1)}}}

\newcommand{\aones}[1]{\hyperref[def:a1s]{{\normalfont(A1-{\ensuremath{{#1}})}}}}

\newcommand{\atwo}{\hyperref[def:a2]{{\normalfont(A2)}}}

\date{\today}

\begin{document}

\title{Stability of solutions to obstacle problems with generalized Orlicz growth}
\author{Petteri Harjulehto and Arttu Karppinen}
\thanks{A. Karppinen is supported by NCN grant no. 2019/34/E/ST1/00120.}

\begin{abstract}
We consider nonlinear equations having generalized Orlicz growth (also known as Musielak--Orlicz growth). We prove that if differential operators $\mathcal{A}_i$ converge locally uniformly to an operator $\mathcal{A}$, then the sequence of solutions $(u_i)$ has a subsequence converging to solution $u$ of the limit operator in Sobolev and H\"older norms.
\end{abstract}


\maketitle

%%%%%%%%%%%%%%%%%%%%%%%%%%%%%%%%%%%%%%%%%%%%%%%%%%%%%%%%
%%%%%%%%%%%%%%%%%%%%%%%%%%%%%%%%%%%%%%%%%%%%%%%%%%%%%%%%
\section{Introduction}

\medskip
In this article we study structural perturbations to a second order nonlinear non-uniformly elliptic partial differential equation 
\begin{align*}
-\div \mathcal{A}(x,\nabla u) = 0.
\end{align*}
Here $\mathcal{A}$ is controlled by a generalized Orlicz function $\phi$, see Assumptions~\ref{ass:main} and Definition~\ref{def2-1}. There is a vast literature on well-posedness of (partial) differential equations and the main principes are that the solution should 1) exist 2) be unique 3) depend continuously on the given data. In our setting, as stated in Proposition \ref{prop:existence}, properties 1) and 2) are known to be true with natural assumptions. For related results, see e.g. \cite{ChlK21,CreGHW21,GwiSW10}. Often the third condition is considered as regarding perturbations of the given boundary values, right hand side of the equation or an obstacle. To best of our knowledge, this property is largely unexplored with generalized Orlicz growth conditions, but convergence results related to the obstacle have been obtained in \cite{FutS21} for metric measure spaces.

Instead of uncertainty in the given data, we aim to study perturbations in the structural properties of the equation itself. More concretely, we consider a sequence of operators $\mathcal{A}_i$ controlled by distinct functions $\phi_i$ and show that under suitable assumptions, the sequence of solutions converge to the solution of the limit equation. These type of perturbations become relevant for example when these equations are applied to physical modelling. For instance, in \cite{DieHHR11,GwiSW10} non-Newtonian fluids have a stress tensor or a pressure term governed by a function belonging to a generalized Orlicz space. In principle, it is not possible to measure the properties of these fluids arbitrarily well, so the resulting equation itself is only an approximation. Our aim is to show, that small perturbations in the equations result in only small changes to the solution.

Equations controlled by generalized Orlicz functions have had wide attraction lately both in applications and in regularity theory. For a general overview, we point the reader to \cite{Chl18}.
 One of the appeals to study problems in generalized Orlicz spaces is to unify theory of many distinctly studied frameworks, such as polynomial, variable exponent, double phase, Orlicz and many others, see for example \cite{BaaB22,ByuL21,DeFM19,EleHL11,EleHL13,Lie91,RagT20}. 
 As already mentioned, existence and uniqueness have been studied widely in generalized Orlicz spaces, but properties of solutions are vigorously studied, too \cite{BenK20,ChlK21,ChlZ21,HasO_pp18,KarL21}. 

Study of stability with respect to equation goes back to Linqvist \cite{Lin87} in the classical $p$-Laplacian case. Afterwards, Li and Martio extended the results and introduced a possibility of an obstacle function~$\psi$ \cite{LiM98}. See also a related recent study of minimizers in metric measure spaces \cite{NasC22}. These results have been generalized to a variable exponent situation and the overarching methods are robust enough to allow for more general growth conditions. To best of our knowledge, these results are already new in Orlicz, double phase and other special cases such as Orlicz double phase, variable exponent Orlicz etc.

We show two convergence results of solutions, first in a Sobolev space and after that in a H\"older space.

\begin{thm}\label{thm:equation-stability}
Let $\Omega$ be a bounded domain, which satisfies measure density condition, Definition~\ref{defn:measure-density}, let Assumptions \ref{ass:main} hold,
and suppose the boundary value function $f$ and the obstacle function $\psi$ belong to $W^{1, \phi^{1+s}}(\Omega)$ for some $s >0$. Then there exists $\delta_0>0$ such that for every ${\delta < \delta_0}$ sequence $(u_i)$ of 
solutions to the $(\mathcal{A}_i, \mathcal{K}_\psi^{f,\phi_i}(\Omega))$-obstacle problem
 has a subsequence which converges in $W^{1, \phi^{1+\delta}}(\Omega)$ to the solution $u$ to the $(\mathcal{A}, \mathcal{K}_\psi^{f,\phi}(\Omega))$-obstacle problem.
\end{thm}

\begin{cor}\label{cor:holder}
Under the assumptions of Theorem~\ref{thm:equation-stability}  and H\"older continuity of the obstacle $\psi$, $(u_i)$ has a subsequence that converges
in $C^{0, \alpha}_{\loc}(\Omega)$, for some $\alpha>0$,  to the solution the $u$ of the $(\mathcal{A},\mathcal{K}_\psi^{f,\phi}(\Omega))$-obstacle problem.
\end{cor}

Let us mention a few key points of our proof. The higher integrability results of the gradient of a solution is the foundation of this method. In the variational integral setting, they are developed in \cite{Kar21,HarHK17}. Thus it is enough to show that also solutions satisfy the required Caccioppoli inequalities to obtain that $\phi_i(x,|\nabla u_i|) \in L^{1+\gamma}(\Omega)$, where $\gamma$ is independent of $i$.  As we do not assume any type of monotonicity of the functions $(\phi_i)$, the underlying Sobolev space is not a priori fixed. However, from higher integrability, we find a common Sobolev space for the tail of the sequence $(u_i)$. In this space we use the utilize the functional analytic properties, namely reflexivity, to find a function to which a subsequence of $(u_i)$ converges. Then we show that this function is in fact the solution to the limit equation.
In a similar manner, for Corollary~\ref{cor:holder} the H\"older regularity results of solutions are the corner stone  of the proof.

\section{Properties of generalized $\Phi$-functions}

Throughout the paper we always consider a 
\textbf{bounded domain $\Omega \subset \Rn$}, i.e.\ a bounded, open and connected 
set. By $p':=\frac p {p-1}$ we denote the H\"older conjugate exponent 
of $p\in [1,\infty]$. The notation $f\lesssim g$ means that there exists a constant
$c>0$ such that $f\le c g$. The notation $f\approx g$ means that
$f\lesssim g\lesssim f$. 
By $c$ we denote a generic constant whose
value may change between appearances.
A function $f$ is \textit{almost increasing} if there
exists $L \ge 1$ such that $f(s) \le L f(t)$ for all $s \le t$
(more precisely, $L$-almost increasing).
\textit{Almost decreasing} is defined analogously.
By \textit{increasing} we mean that this inequality holds for $L=1$ 
(some call this non-decreasing), similarly for \textit{decreasing}. By $\rightharpoonup$ we denote the weak convergence.

By $C^{0,\alpha}_\loc(\Omega)$ we mean a space of locally H\"older continuous functions. In a compact subset $K\subset \Omega$ we define its norm as
\begin{align*}
\|f\|_{C^{0,\alpha}(K)} = \sup_{x \in K} |f(x)| + \sup_{\substack{x \not =y \\ x,y \in K}} \dfrac{|f(x)-f(y)|}{|x-y|^{\alpha}}.
\end{align*}

\subsection{Generalized $\Phi$-functions}

\begin{defn}
\label{def2-1}
We say that $\phi: \Omega\times [0, \infty) \to [0, \infty]$ is a 
\textit{weak $\Phi$-function}, and write $\phi \in \Phiw(\Omega)$, if 
the following conditions hold for almost every $x\in\Omega$:
\begin{itemize}
\item 
for every measurable function $f:\Omega\to \R$ the function $y \mapsto \phi(y, f(y))$ is measurable;
\item
the function $t \mapsto \phi(x, t)$ is non-decreasing;
\item 
$\displaystyle \phi(x, 0) = \lim_{t \to 0^+} \phi(x,t) =0$ and $\displaystyle \lim_{t \to \infty}\phi(x,t)=\infty$;
\item 
the function $t \mapsto \frac{\phi(x, t)}t$ is $L$-almost increasing on $(0,\infty)$ with $L$ independent of $x$.
\end{itemize}
If $\phi\in\Phiw(\Omega)$ is additionally convex and left-continuous with respect to $t$ for almost every $x$, then $\phi$ is a 
\textit{convex $\Phi$-function}, and we write $\phi \in \Phic(\Omega)$. 
\end{defn}

Let us consider $\phi \in \Phiw(\Omega)$. By $\phi^{-1}$ we mean a generalized inverse, defined by
\begin{align*}
\phi^{-1}(x,\tau) = \inf \{t \geq 0 : \phi(x,t) \geq \tau\}.
\end{align*}
We say that $\phi$ satisfies 
\begin{itemize}
\item[(A0)]\label{def:a0}
if there exists $\beta \in(0, 1]$ such that $ \beta \le \phi^{-1}(x,
1) \le \frac1{\beta}$ for a.e.\ $x \in \Omega$, or equivalently there exists $\beta \in(0,1]$ such that 
$\phi(x,\beta) \le 1\le \phi(x,1/\beta)$ for almost every $x \in \Omega$;
\item[(A1)]\label{def:a1}
if there exists $\beta\in (0,1)$ such that,
for every ball $B$ and a.e.\ $x,y\in B \cap \Omega$,
\[
\beta \phi^{-1}(x, t) \le \phi^{-1} (y, t) 
\quad\text{when}\quad 
t \in \bigg[1, \frac{1}{|B|}\bigg];
\]
\item[(A2)]\label{def:a2}
if there exists $\phi_\infty \in \Phiw$, $h \in L^1(\Omega)\cap L^\infty(\Omega)$, $\beta \in(0, 1]$ and $s>0$ such that
\[
\phi(x, \beta t) \le \phi_\infty(t) + h(x) \quad \text{ and } \quad \phi_\infty(\beta t) \le \phi(x, t) + h(x) 
\] 
for almost every $x \in \Omega$ when $\phi_\infty(t) \in [0, s]$ and $\phi(x, t) \in [0,s]$, respectively.
\item[(aInc)$_p$] \label{def:aInc} if
$t \mapsto \frac{\phi(x,t)}{t^{p}}$ is $L_p$-almost 
increasing in $(0,\infty)$ for some $L_p\ge 1$ and a.e.\ $x\in\Omega$;
\item[(aDec)$_q$] \label{def:aDec}
if
$t \mapsto \frac{\phi(x,t)}{t^{q}}$ is $L_q$-almost 
decreasing in $(0,\infty)$ for some $L_q\ge 1$ and a.e.\ $x\in\Omega$.
\end{itemize} 

We write \ainc{} in the case there exists $p>1$ such that $\phi$ satisfies \ainc{p}. Similarly, \adec{} means that there exists $q< \infty$ such that $\phi$ satisfies \adec{q}.
By \cite[Lemma 2.2.6]{HarH_book} doubling $t\mapsto \phi(x,t)$ for almost every $x \in \Omega$ and \adec{} are equivalent. Doubling and increasingness also imply the following triangle inequality $\phi(x,t+s) \lesssim \phi(x,t) + \phi(x,s)$.

$\phi^{\ast}$ is the conjugate $\Phi$-function defined as
\begin{align*}
\phi^{\ast}(x,t) = \sup_{s>0} \{st - \phi(x,s)\},
\end{align*}
and it gives the \emph{Young's inequality}:
\[
st \le \phi(x,s) + \phi^{\ast}(x,t).
\]
The conjugate function enjoys the following properties:
\begin{itemize}
\item[1)] $\phi^{\ast}$ satisfies \ainc{q'} and \adec{p'} if and only if $\phi$ satisfies \ainc{p} and \adec{q}, respectively \cite[Proposition 2.4.9]{HarH_book};
\item[2)] If $\phi$ satisfies \ainc{} and \adec{}, then $\phi^{\ast}\left (x,\dfrac{\phi(x,t)}{t}\right ) \lesssim \phi(x,t)$ (see below);
\item[3)] $\phi^{\ast}$ satisfies \azero{}, \aone{} and \atwo{} if and only if $\phi$ satisfies them \cite[Proposition 2.4.5, Lemma 3.7.6, Lemma 4.1.7, Lemma 4.2.4]{HarH_book}.
\end{itemize}
For $\phi \in \Phi_c(\Omega)$ we have $\phi^{\ast}\left (x,\dfrac{\phi(x,t)}{t}\right ) \leq \phi(x,t)$ \cite[p. 35]{HarH_book}. In the case of weak $\Phi$-function we choose an equivalent convex $\Phi$-function with an equivalence constant $L$ \cite[Lemma 2.2.1]{HarH_book}. Then also $\eta^{\ast} \simeq \phi^{\ast}$ \cite[Lemma 2.4.3 (c)]{HarH_book} with the same constant $L$. Now using \adec{} of $\phi^*$ and $\phi$ we get
\begin{align*}
\phi^{\ast}\left (x,\dfrac{\phi(x,t)}{t}\right ) &\lesssim \phi^{\ast}\left (x,  \dfrac{\phi(x,t)}{L^2t}\right ) \leq \eta^{\ast}\left (x,  \dfrac{\eta(x,Lt)}{Lt}\right ) = \eta(x,Lt) \leq \phi(x,L^2t) \lesssim \phi(x,t).
\end{align*}

\subsection{Generalized Orlicz space}

Let $\phi \in \Phiw(\Omega)$. The \emph{generalized Orlicz space} is defined as the 
set 
$$L^{\phi}(\Omega) : = \bigg\{ f \in L^0(\Omega) : \lim_{\lambda\to 0^+} 
\varrho_{\phi}(\lambda f) =0 \bigg\}
$$
equipped with the (Luxemburg) (quasi)norm
$$ \| f \|_{L^{\phi}(\Omega)} := \inf \bigg\{ \lambda >0 : \varrho_{\phi} 
\Big( \frac{ f}{\lambda}\Big) \le1\bigg\},$$
where $ \varrho_{\phi}(f)$ is the modular of $f \in L^0(\Omega)$ defined by 
\[
\varrho_{\phi}(f):=\int_{\Omega} \phi(x, |f(x)|)\,dx. 
\]

A function $u \in L^{\phi}(\Omega)$ belongs to the \emph{Orlicz--Sobolev space 
$W^{1,\phi}(\Omega)$} if its weak partial derivatives $\partial_1 u, \dots 
\partial_n u$ exist and belong to $L^{\phi}(\Omega)$. The norm of this space is defined as
\begin{align*}
\|u\|_{W^{1,\phi}(\Omega)} := \|u\|_{L^{\phi}(\Omega)} + \||\nabla u|\|_{L^{\phi}(\Omega)}.
\end{align*}
With a slight abuse of notation, we abbreviate $\||\nabla u|\|_{L^{\phi}(\Omega)}$ as $\|\nabla u\|_{L^{\phi}(\Omega)}$.
We also omit the set if there is no danger of confusion, i.e. $\|u\|_{L^{\phi}(\Omega)} = \|u\|_{\phi}$.
Furthermore, 
$W^{1,\phi}_0 (\Omega)$ is the closure of $C^{\infty}_0(\Omega)$ in 
$W^{1,\phi}(\Omega)$. $C^\infty(\Omega) \cap W^{1,\phi}(\Omega)$ is dense in $W^{1,\phi}(\Omega)$ for bounded $\Omega$ if $\phi$ satisfies \azero{}, \aone{} and \adec{}, see \cite[Theorem 6.4.7]{HarH_book}.

\subsection{Auxiliary results}

\begin{defn}\label{defn:measure-density}
We say that a domain $\Omega\subset \Rn$ satisfies \emph{measure density condition} if
there exists $c, r_0>0$ such that
\begin{equation*}
|B(z,r) \cap \Omega^c| \ge c |B(z, r)|
\end{equation*}
for all $z \in \partial \Omega$ and all $0<r\le r_0$.
\end{defn}

 Since in the following lemma we require boundedness of the Hardy--Littlewood maximal function in $L^{\phi}(\Rn)$, we need to introduce the assumption \atwo{}. This assumption is inherited by our main theorem. In general, it is not needed with problems having a bounded domain $\Omega$ and even though higher integrability (Proposition \ref{pro:higher-integrability}) requires $\phi$ to be defined in $\Rn$, it is not required. 
 
\begin{lem}\label{Hardy_zbv}
Let $\phi \in \Phiw(\Rn)$ satisfy \azero{}, \aone{}, \atwo{}, \ainc{} and \adec{}.
  Let $u\in W^{1,\phi}(\Omega)$, and extend u by zero outside $\Omega$. If
  \[
  \frac{u}{\dist(\cdot,\partial\Omega)} \in L^{\phi}(\Rn),
  \]
  then  $u\in W^{1,\phi}_0(\Omega)$
\end{lem}
\begin{proof}
  The proof of \cite[Theorem 3.4, p.  225]{EdmE87} can be
  adapted to cover the generalized Orlicz case, since all we need to assume is that $C^{\infty}_0(\Rn)$ is dense in $W^{1,\phi}(\Rn)$ \cite[Theorem 6.4.4]{HarH_book}, convolution converges in $L^\phi(\Rn)$ \cite[Theorem 4.4.7]{HarH_book}, and that the Hardy--Littlewood maximal function is bounded from $L^{\phi}(\Rn)$ to $L^{\phi}(\Rn)$, see \cite{Has15,Has16,HarH_book}. 
\end{proof}

Next we prove Hardy's inequality. 

\begin{lem}[Hardy's inequality]\label{lem:Hardy}
Let $\Omega \subset \Rn$ be bounded
 and assume the measure density condition, Definition~\ref{defn:measure-density}, is satisfied.
Let $\phi \in \Phiw(\Omega)$ satisfy \azero{}, \aone{}, \ainc{} and \adec{}.
Then  
\begin{equation*}
    \Bigg\| \frac{u}{\dist(\cdot,\partial
      \Omega)}\Bigg\|_{L^{\phi}(\Omega)} 
    \leq C \|\nabla u\|_{L^{\phi}(\Omega)}
  \end{equation*}
  holds for all functions $u\in W_0^{1, \phi}(\Omega)$. 
\end{lem}

\begin{proof}
 Since $C_0^\infty(\Omega)$ functions are dense in $W^{1, \phi}_0(\Omega)$, it suffices to prove the claim for them.
 By \cite[Proposition 1]{Haj99} we have
 \[
 |u(x)| \le C \dist(x,\partial \Omega) M|\nabla u| (x)
 \]
for all $u \in C_0^\infty(\Omega)$.
 This yields the claim since $M: L^\phi(\Omega) \to L^\phi(\Omega)$ is bounded \cite[Theorem 4.3.4]{HarH_book}.
\end{proof}

The next lemma has been proved under capacity fatness in \cite[Lemma 3.25]{LiM98} to a classical case and in
\cite[Lemma 4.7]{EleHL11} to the variable exponent case.

\begin{lem}\label{lemma:LM} 
Let $\Omega \subset \Rn$ be bounded, and assume that it satisfies the measure density condition, Definition~\ref{defn:measure-density}.
Let  $\phi \in \Phiw(\Rn)$ and
let $\phi_i \in \Phiw(\Rn)$ be a sequence of  $\Phi$-functions such that for every sequence $(t_i)$ from  
$[0, \infty]$ converging to $t$  and for almost every $x \in \Omega$,  we have $\phi_i(x,t_i) \to \phi(x, t)$.
  Assume that  $\phi$ and each $\phi_i$ satisfies \azero{}, \aone{}, \atwo{}, \ainc{} and \adec{} with constants independent of $i$.
  
 Let $u\in W^{1,\phi}(\Omega)$ and $u_i\in
  W^{1,\phi_i}(\Omega)$ for every $i$ so that $u_i \to u$ almost
  everywhere in $\Omega$. If $f \in W^{1,\phi}(\Omega) \cap
  \bigcap_i W^{1,\phi_i}(\Omega)$ and $u_i-f \in
  W^{1,\phi_i}_0(\Omega)$ for every $i$ with
  \begin{equation*}
    \int_\Omega \phi_i(x, |\nabla(u_i-f)|) \,dx \leq M,
  \end{equation*}
  where $M$ is finite and independent of $i$, then $u-f \in
  W^{1,\phi}_0(\Omega)$.
\end{lem}

\begin{proof}
   Applying Hardy's inequality (Lemma~\ref{lem:Hardy}) we obtain
  \begin{equation}\label{Hardyey2}
    \Bigg\| \frac{|u_i -f|}{\dist(\cdot,\partial
      \Omega)}\Bigg\|_{L^{\phi_i}(\Omega)} 
    \leq C \|\nabla (u_i -f)\|_{L^{\phi_i}(\Omega)}.
  \end{equation}
Since $\phi_i$ satisfies \adec{}, the norm is finite if and only if the modular is finite. 
Thus by the assumption we have
\[
\int_\Omega \phi_i\Bigg(x, \frac{|u_i-f|}{\dist(x,\partial \Omega)} \Bigg) \,dx \leq
  c(M).
\]   
By assumptions   
$\phi_i\big(x, \frac{|u_i-f|}{\dist(x,\partial \Omega)} \big) \to \phi\big(x, \frac{|u - f|}{\dist(x,\partial \Omega)}\big)$ 
and hence Fatou's lemma gives
  \[
  \int_\Omega \phi\Bigg(x, \frac{|u - f|}{\dist(x,\partial \Omega)}
  \Bigg) \, dx \leq \liminf_{i\to \infty}\int_\Omega \phi_i\Bigg(x, \frac{|u_i
    -f|}{\dist(x,\partial \Omega)} \Bigg) \,dx \leq
  c(M).
  \]
  Since $u-f \in W^{1,\phi}(\Omega)$, the claim follows by
  Lemma~\ref{Hardy_zbv}.
\end{proof}

\section{Properties of solutions}

We start by laying out our assumptions required to prove stability. We show below that these assumptions are essentially equivalent to the assumptions given in the variable exponent situation \cite{EleHL13}.

\begin{assumptions}\label{ass:main}
For $i=1,2\dots$, let $\phi, \phi_i \in \Phic(\Rn)$  satisfy \azero{}, \aone{}, \atwo{}, \ainc{} and \adec{q} with $q<n$ and 
constants independent of $i$.
 
Suppose that $\mathcal{A},\mathcal{A}_i: \Omega \times \Rn \to \Rn$ are Carath\'eodory functions, i.e. $x\mapsto \mathcal{A}(x,\xi)$ is measurable for every $\xi$ and $\xi \mapsto \mathcal{A}(x,\xi)$ is continuous for almost every $\xi \in \Rn$. Suppose additionally
\begin{itemize}
\item[1)] There exists a constant $c_1>0$  independent of $i$ such that 
\begin{align*}
c_1 \phi_i(x, |\xi|) \leq \mathcal{A}_i(x,\xi)\cdot \xi ;
\end{align*}
\item[2)]There exists a constant $c_2>0$  independent of $i$ such that 
\begin{align*}
|\mathcal{A}_i(x,\xi)| \leq c_2 \dfrac{\phi_i(x,|\xi|)}{|\xi|};
\end{align*}
\item[3)]For all $x \in \Omega$ and $\xi \neq \eta$ we have
\begin{align*}
(\mathcal{A}_i(x,\xi) - \mathcal{A}_i(x,\eta))\cdot(\xi -\eta) >0;
\end{align*}
\end{itemize}

Let $\gamma >0$ be from the higher integrability of solutions, see Proposition~\ref{pro:higher-integrability}. We make three assumptions 
about convergence of $(\mathcal{A}_i)$ and $(\phi_i)$:
\begin{itemize}
\item[4)] $\mathcal{A}_i(x, \xi ) \to \mathcal{A}(x, \xi)$ locally uniformly i.e. the convergence is uniform for every $D \times (0, t)$, where $D \Subset \Omega$ and $t\in (0, \infty)$;
\item[5)] $\phi_i(x, t) \to \phi(x, t)$ for all $x \in \Omega$, and for all $t>0$  and $(\phi_i)$ is locally uniformly equicontinous with respect to the $t$-variable;
\item [6)] there exist $L>0$ and $t_0 >0$ such that for every  $\theta \in (0, \gamma)$ there exists $i_\theta \in \N$ such that
\[
\frac{\phi_i(x, t)}{\phi(x, t)^{1+ \theta}}  \le L \text{ for all $i\ge i_\theta$, for all $x \in \Omega$ and for all $t \ge t_0$},
\]
and 
\[
\frac{\phi(x, t)}{\phi_i(x, t)^{1+ \theta}}  \le L \text{ for all $i\ge i_\theta$, for all $x \in \Omega$ and for all $t \ge t_0$}.
\]
\end{itemize}
\end{assumptions}

We notice that structural conditions 1-3) for $\mathcal{A}_i$ and the locally uniform convergence imply similar structural conditions for $\mathcal{A}$ also: For example, since
\begin{align*}
\mathcal{A}(x,\xi)\cdot \xi = (\mathcal{A}(x,\xi) - \mathcal{A}_i(x,\xi))\cdot \xi + \mathcal{A}_i(x,\xi)\cdot \xi,
\end{align*}
where the first term on the right hand side converges to 0 due to pointwise convergence, we obtain
\begin{align*}
\mathcal{A}_i(x,\xi)\cdot \xi \geq c_1 \phi_i(x,|\xi|) \to c_1 \phi(x,|\xi|).
\end{align*}
Other conditions are proven similarly.

Condition 6) is inspired by the concept of dominations in Orlicz spaces used for instance in \cite{Cia04}. It is said that an Orlicz function $B$ dominates $A$ if there exists a constant $c>0$ such that $A(t) \leq B(ct)$. Since we are dealing with doubling functions, this condition can be rewritten as $\frac{A(t)}{B(t)} \leq \tilde c$. Now 6) just states that increasing the exponent by $1+\theta$ yields domination in both direction at infinity (i.e. for large $t$). 

\begin{eg}
Let $\phi_i(x, t) := t^{p_i(x)}$.
In the variable exponent case stability has been studied in \cite{EleHL13}, see also \cite{EleHL11},  under assumptions that  $\sup_i \sup_{x \in \Omega }p_i(x) <n$,  $(p_i)$ converges 
pointwise to $p$ and 
each $p_i$ is $\log$-Hölder continuous with the same constant. 
 We show that our  assumptions are essentially equivalent in the variable exponent case.  First of all, assumptions 1-4) are identical in both cases.

 Assumptions \ref{ass:main} imply ones in \cite{EleHL13}: By \cite[Lemma 7.1.2]{HarH_book} $\phi_i$ satisfies \aone{} if and only if $\frac1{p_i}$ is locally $\log$-Hölder continuous. 
Since each $p_i$ is uniformly bounded this yields that 
$\phi_i$ satisfies \aone{} if and only if $p_i$ is locally $\log$-Hölder continuous. 

Assumptions in \cite{EleHL13} imply Assumptions \ref{ass:main}: Since $p_i \to p$ and $(p_i)$ is uniformly bounded by $n$, condition 5) follows. We have to show that 6) holds. Since $p_i \to p$ pointwise and $(p_i)$ is equicontinuous, we obtain by 
Arzel\`{a}--Ascoli theorem that $p_i \to p$ uniformly. Thus for every $\theta>0$ there exists $i_\theta$ such that
$|p(x) - p_i(x)|<\theta$ for all $x \in \Omega$ and all $i>i_\theta$. Thus for $t \ge 1$ we have
$t^{p_i(x)} \le t^{p(x) + \theta} \le t^{(1+\theta)p(x)} $ and $t^{p(x)} \le t^{p_i(x) + \theta}\le t^{(1+\theta)p_i(x)}$, and hence 6) holds. 
Note that in variable exponent spaces, extension of $p$ from $\Omega$ to $\Rn$ can always be done unlike in generalized Orlicz spaces. The extended exponent satisfies the decay $\log$-Hölder condition, that is slightly stronger than \atwo{}, see  \cite[Section 7.1]{HarH_book}.
\end{eg}

Next we define solutions to an obstacle problem. 

\begin{defn}
Let $f\in W^{1,\phi}(\Omega)$ be a boundary value function and $\psi \in W^{1,\phi}(\Omega)$ be an obstacle. Denote 
\begin{align*}
\mathcal{K}_{\psi}^{f,\phi}(\Omega) := \{u \in W^{1,\phi}(\Omega) : u-f \in W^{1,\phi}_0(\Omega), u \geq \psi \text{ a.e. in } \Omega\}.
\end{align*}
We say that  $u$ is a solution to an $(\mathcal{A}, \mathcal{K}_\psi^{f,\phi}(\Omega))$-obstacle problem  if
\begin{align}\label{eq:equation}
\int_{\Omega} \mathcal{A}(x, \nabla u)\cdot \nabla (w-u) \, dx \geq 0
\end{align}
for all $w \in \mathcal{K}_\psi^{f,\phi}(\Omega)$.
\end{defn}
We immediately see that $u$ is also a solution in any open subset of $\Omega$.

We say a function $u\in W^{1,\phi}_\loc(\Omega)$ is an $\mathcal{A}$-supersolution if
\begin{align*}
\int_{\Omega} \mathcal{A}(x,\nabla u) \cdot \nabla w \, dx \geq 0
\end{align*}
for all non-negative $w \in W^{1,\phi}_0(\Omega)$.
Since for an solution $u$ to an $(\mathcal{A}, \mathcal{K}_\psi^{f,\phi}(\Omega))$-obstacle problem, we have $u+w \in \mathcal{K}_{\psi}^{f,\phi}(\Omega)$ when $w \in W^{1,\phi}_0(\Omega)$ is non-negative, $u$ is also a supersolution.
Next we give a proof for global higher integrability of a solutions to obstacle problems. 

\begin{lem}[Interior Caccioppoli inequality]
Let $\phi \in \Phiw(\Omega)$ satisfy \adec{q} and let $u$ be a solution to an $(\mathcal{A}, \mathcal{K}_\psi^{f,\phi}(\Omega))$-obstacle problem, where $\mathcal{A}$ satisfies 1-3) in Assumptions \ref{ass:main}. Then
\begin{align*}
\fint_{B} \phi(x,|\nabla u|) \, dx &\lesssim \fint_{B}\phi\left (x, \dfrac{|u-u_{2B}|}{\diam(2B)}\right ) \, dx + \fint_{B}\phi\left (x, \dfrac{|\psi-\psi_{2B}|}{\diam(2B)}\right ) \, dx  \\
&\quad +\fint_{B}\phi\left (x, |\nabla \psi| \right ) \, dx
\end{align*}
for any ball $B$ with $2B \subset \Omega$. The implicit constant depends only on $n,c_1,c_2,q$ and $L_q$.
\end{lem}

\begin{proof}
Let $\eta\in C^{\infty}_0(2B)$ be a cutoff function such that, $0\le \eta\le 1$, $\eta=1$ in $B$ and $|\nabla \eta| \lesssim \frac{1}{r}$, 
where $r$ is the radius of $B$. We choose $w := u - u_{2B} -(u - u_{2B} -(\psi - \psi_{2B}))\eta^q$ as our test function.
 Note that $w \in \mathcal{K}_{\psi-u_{2B}}^{f-u_{2B},\phi}(\Omega)$ since $w-(f-u_{2B})\in W^{1,\phi}_0(\Omega)$ and $w \geq \psi-u_{2B}$ 
 due to $u \geq \psi$ and $u_{2B} \geq \psi_{2B}$. 
 Since $u- u_{2B}$ is a solution to the $(\mathcal{A}, \mathcal{K}_{\psi-u_{2B}}^{f-u_{2B},\phi}(\Omega))$-obstacle problem, we obtain 
 by \eqref{eq:equation} that
\begin{align*}
\int_{2B} \mathcal{A}(x,\nabla u) \cdot \nabla u \, dx &\leq \int_{2B} \mathcal{A}(x,\nabla u) \cdot \nabla w \, dx \leq \int_{2B} (1-\eta^q)\mathcal{A}(x,\nabla u) \cdot \nabla u \, dx \\
&\quad + \int_{2B}\eta^q \mathcal{A}(x,\nabla u) \cdot \nabla \psi \,dx \\
&\quad + q\int_{2B} (\eta^{q-1} (|\psi-\psi_{2B}| + |u - u_{2B}|))\mathcal{A}(x,\nabla u) \cdot \nabla \eta \, dx.
\end{align*}
Using the structural conditions and properties of $\eta$ we arrive at
\begin{align*}
c_1\int_{2B}\phi(x,|\nabla u|) \eta^q\, dx &\leq c_2\int_{2B}\dfrac{\phi(x,|\nabla u|)}{|\nabla u|} |\nabla \psi| \eta^q \, dx \\
&\quad+c_2q\int_{2B}\eta^{q-1}\dfrac{\phi(x,|\nabla u|)}{|\nabla u|}  \left (\frac{|\psi-\psi_{2B}|}{r} + \frac{|u - u_{2B}|}{r}\right) \, dx .
\end{align*}
Let $c$ be a constant from $\phi^*(x, \phi(x, t)/t) \le c \phi(x, t)$.
We use Young's inequality for the both terms on the right hand side, and then  \ainc{q'} of $\phi^{\ast}$ and \adec{} of $\phi$ to get 
\begin{align*}
&c_1 \int_{2B}\phi(x,|\nabla u|) \eta^q \, dx \\
& \quad \le c_2\int_{2B} \bigg(\phi^*\Big(x,  \Big(\frac{c_1}{4 L_q c c_2}\Big)^{\frac{q-1}{q}} \dfrac{\phi(x,|\nabla u|)}{|\nabla u|} \Big) +\phi\Big (x, \Big(\frac{4 L_q c c_2}{c_1}\Big)^{\frac{q-1}{q}} |\nabla \psi| \Big) \bigg) \eta^q \, dx \\
&\qquad+c_2q\int_{2B}\phi^{\ast}\left (\Big(\frac{c_1}{4 q L_q c c_2}\Big)^{\frac{q-1}{q}} \dfrac{\phi(x,|\nabla u|)\eta^{q-1}}{|\nabla u|}\right )\\
&\qquad + \phi\left (x, \Big(\frac{4 qL_q c c_2}{c_1}\Big)^{\frac{q-1}{q}} \bigg(\frac{|\psi-\psi_{2B}|}{r} + \frac{|u - u_{2B}|}{r}\bigg)\right) \, dx \\
&\leq \int_{2B} \dfrac{c_1}{2} \phi(x,|\nabla u|) \eta^q\, dx +  C\int_{2B}\phi(x,|\nabla \psi|) \, dx\\
&\quad + C \int_{2B} \phi\left (x,  \dfrac{|u-u_{2B}|}{r}\right )+\phi\left (x, \dfrac{|\psi-\psi_{2B}|}{r}\right ) \, dx.
\end{align*} 
Note that in the previous estimate we exclude $\eta^q$ from Young's inequality in the first term on the righ hand side, while in the second term we include $\eta^{q-1}$ to Young's inequality to correct its exponent.  
Subtracting the first term on the right hand side, using \adec{} to change from radius to diameter, and using $\eta = 1$ in $B$  we come to \begin{equation*}
\fint_{B}\phi(x,|\nabla u|) \, dx \lesssim \fint_{2B}\phi\left (x, \dfrac{|u-u_{2B}|}{\diam(2B)}\right )+\phi\left (x, \dfrac{|\psi-\psi_{2B}|}{\diam(2B)}\right )+ \phi(x,|\nabla \psi|)  \, dx. \qedhere
\end{equation*}
\end{proof}

Next we prove a similar result for balls which are sufficiently near the boundary of $\Omega$. Note that in the next result, assumptions \azero{} and \aone{} are only needed for the case where $f \geq \psi$ almost everywhere does not hold a priori.

\begin{lem}[Caccioppoli inequality near the boundary] Let $\phi \in \Phiw(\Rn)$ satisfy \azero{}, \aone{} and \adec{q}, $u$ be a solution to an
$(\mathcal{A}, \mathcal{K}_\psi^{f,\phi}(\Omega))$-obstacle problem \eqref{eq:equation},  where $\mathcal{A}$ satisfies 1-3) in Assumptions \ref{ass:main}, and $f \in W^{1,\phi}(\Omega)$ be such that $u-f \in W^{1,\phi}_0(\Omega)$. Then
\begin{align*}
\dfrac{1}{|B|}\int_{B \cap \Omega} \phi(x,|\nabla u|) \, dx \lesssim  \dfrac{1}{|2B|} \int_{2B \cap \Omega} \phi(x, |\nabla f|) \, dx 
+\dfrac{1}{|2B|} \int_{2B \cap \Omega} \phi\left (x, \dfrac{|u-f|}{\diam(2B)}\right ) \, dx, 
\end{align*}
where $B$ is a ball with center in $\Omega$ but $2B \cap \Omega^{c} \not= \emptyset$. The implicit constant depends only on $n,c_1,c_2,q$ and $L_q$.
\end{lem}

\begin{proof}
We proceed similarly as in the interior Caccioppoli inequality.  We take a cutoff function $\eta \in C^{\infty}_0(2B)$ with $0 \le \eta\le 1$, $\eta=1$ in $B$ and $|\nabla \eta| \lesssim \frac{1}{r}$.
For the test function we choose $w := u-\eta^{q}(u-f) \in W^{1,\phi}_0(2B\cap \Omega)$. Without loss of generality, we may assume $f \geq \psi$ in $\Omega$ and thus $w \in \mathcal{K}_{\psi}^{f,\phi}(\Omega)$ \cite[Lemma 4.6]{Kar21}.
 Then
\begin{align*}
\int_{\Omega}  \mathcal{A}(x, \nabla u) \cdot \nabla u \, dx &\leq \int_{ \Omega} \mathcal{A}(x, \nabla u) \cdot \nabla w \, dx \leq \int_\Omega (1-\eta^q)\mathcal{A}(x,\nabla u)\cdot \nabla u \, dx \\
&\quad - \int_{2B \cap \Omega} q \eta^{q-1}(u-f)\mathcal{A}(x,\nabla u) \cdot \nabla \eta \, dx \\
&\quad + \int_{2B \cap \Omega} \eta^{q} \mathcal{A}(x,\nabla u) \cdot \nabla f \,dx,
\end{align*}
from which we get with structural conditions  together with $\eta^q \leq 1$ that
\begin{align*}
c_1 \int_{2B\cap \Omega} \phi(x, |\nabla u|)\eta^{q} \, dx &\leq c_2q \int_{2B \cap \Omega} \dfrac{\phi(x,|\nabla u|)}{|\nabla u|}\eta^{q-1} |u-f| |\nabla \eta| \, dx \\
&\quad + c_2\int_{2B\cap \Omega} \dfrac{\phi(x,|\nabla u|)}{|\nabla u|} |\nabla f| \,dx.
\end{align*}
We use Young's inequality similarly as in the previous lemma to get
\begin{align*}
c_1 \int_{2B\cap \Omega} \phi(x, |\nabla u|) \eta^{q} \, dx &\leq \dfrac{c_1}{2} \int_{2B\cap \Omega} \phi(x,|\nabla u|) \eta^{q}\, dx + C\int_{2B\cap \Omega} \phi(x,|\nabla f|) \, dx \\
&\quad + C\int_{2B\cap \Omega} \phi\left (x, \dfrac{|u-f|}{r}\right ) \, dx .
\end{align*}
Finishing as in the interior version, we see that
\begin{equation*}
\dfrac{1}{|B|}\int_{B \cap \Omega} \phi(x,|\nabla u|) \, dx \lesssim  \dfrac{1}{|2B|} \int_{2B \cap \Omega} \phi(x, |\nabla f|) \, dx+ \dfrac{1}{|2B|} \int_{2B \cap \Omega} \phi\left (x, \dfrac{|u-f|}{\diam(2B)}\right ) \, dx.
\end{equation*}
\end{proof}

Combining the Caccioppoli inequalities with Sobolev--Poincar\'e inequality \cite[Proposition 6.3.12]{HarH_book} we get the global higher integrability exactly the same way as in proof of \cite[Theorem 1.2]{Kar21}.

\begin{prop}\label{pro:higher-integrability}
Let $\phi \in \Phiw(\Rn)$ satisfy \azero{}, \aone{}, \ainc{p} and \adec{q}. Suppose the measure density condition, Definition~\ref{defn:measure-density}, is satisfied and $u$ is a solution to an $(\mathcal{A}, \mathcal{K}_\psi^{f,\phi}(\Omega))$-obstacle problem,  where $\mathcal{A}$ satisfies 1-3) in Assumptions \ref{ass:main}. If $\phi(x,|\nabla f|),\phi(x,|\nabla \psi|) \in L^{1+s}(\Omega)$ for some $s>0$,
then there exists $\gamma_0 \in (0, s]$ such that
\begin{align*}
\int_{\Omega}\phi(x,|\nabla u|)^{1+\gamma} \, dx &\lesssim \left (\int_{\Omega} \phi(x,|\nabla u|) \, dx \right )^{1+\gamma} + \int_{\Omega} \phi(x,|\nabla f|)^{1+\gamma} \, dx \\
&\quad +\int_{\Omega} \phi(x,|\nabla \psi|)^{1+\gamma} \, dx + 1
\end{align*}
for every $0<\gamma < \gamma_0$. Here the implicit constant depends only on constants in \azero{}, \aone{}, \ainc{p} and \adec{q} $n,|\Omega|, c_1,c_2,\gamma$ and the measure density constant. 
\end{prop}

We end this section by recalling the existence and uniqueness properties together with relation of solutions and minimizers. The existence and uniqueness of the solution follows directly from \cite[Theorem 2]{ChlK21}.

\begin{prop}\label{prop:existence}
Let $\phi \in \Phic(\Omega)$ satisfy \azero{}, \aone{}, \ainc{} and \adec{}. If $\mathcal{K}_{\psi}^{f,\phi}(\Omega) \not= \emptyset$, then there exists a unique  solution to $(\mathcal{A}, \mathcal{K}_\psi^{f,\phi}(\Omega))$-obstacle problem.
\end{prop}

We say that $u \in W^{1, \phi}_\loc(\Omega)$ is a  \emph{local quasisuperminimizer} if there exists a constant $Q\ge 1$ such that
\[
\int_{D \cap\{u \neq v\}} \phi(x, |\nabla u|) \, dx \le Q \int_{D \cap\{u \neq v\}} \phi(x, |\nabla v|) \, dx
\]
for all open $D \Subset \Omega$ and all non-negative  $v \in W^{1,\phi}_\loc(D)$ with $u-v \in W^{1, \phi}_0(D)$. Moreover $u$ is a \emph{local quasisubminimizer} if $-u$ is a  local quasisuperminimizer, and $u$ is a \emph{local quasiminimizer} if it is a local quasisuper- and  quasisubminimizer.

Proof for the next result can be found in \cite[Lemma 3.7, Corollary 3.9]{ChlZ21} by recalling that a solution to an obstacle problem is a supersolution.

\begin{lem}\label{lem:sol-is-quasimin}
Let $\phi \in \Phiw(\Omega)$ satisfy \adec{q} and $u$ be  a solution to an $(\mathcal{A}, \mathcal{K}_\psi^{f,\phi}(\Omega))$-obstacle problem. Then $u$ is a local quasisuperminimizer and a quasisuperminimizer with boundary values $f$. 
\end{lem}

\section{Stability with respect to the equation}
In this section we prove our main theorems.

\begin{proof}[Proof of Theorem \ref{thm:equation-stability}]
Let $\gamma \in (0,\gamma_0)$ be from the higher integrability of solutions, see Proposition~\ref{pro:higher-integrability} and
let $\theta \in (0, \gamma)$ from the convergence assumptions. By assumptions we have for $t > t_0$ that
\[
\phi_i(x, t) \le L \phi(x, t)^{1+\theta}. 
\]
Assume then that  $t \in [0, t_0]$. Let $\beta >0$ be from \azero{}. If $t< \beta$ then
$\phi_i(x, t) \le \phi_i(x, \beta) \le 1$. If $t \in (\beta, t_0)$, then we obtain by  \adec{} and \azero{} that
\[
\phi_i(x, t) \le L_q \left (\frac{t}{\beta}\right)^q \phi_i(x, \beta) \le  L_q \left (\frac{t_0}{\beta}\right)^q.
\]
By combining the cases we  have 
\begin{align}\label{eq:exponent}
\phi_i(x,t) \lesssim \phi(x,t)^{1+\theta}  + 1 \quad \text{ and } \quad  \phi(x,t) \lesssim \phi_i(x,t)^{1+\theta}  + 1,
\end{align}
where the later comes similarly. Here also $i \ge i_\theta$,  and the implicit constant depends on $t_0$ i.e. on assumption 6).
Since $\phi$ and all $\phi_i$ satisfy \adec{}, the modular is finite if and only if the norm is finite.
Since $f \in W^{1, \phi^{1+\gamma}}(\Omega)$ and $\Omega$ is bounded, we find that
\[
\int_\Omega \phi_i(x,f) \, dx \lesssim \int_\Omega \phi(x,f)^{1+\theta} +1 \, dx \le c + |\Omega|< \infty,
\]
and similarly for the gradient.
This implies that $f \in W^{1,\phi_i}(\Omega)$ for  all $i>i_\theta$.

\medskip
\noindent \textbf{Claim 1:} \textit{For  $i \in \N$ large enough we have
\begin{align*}
\int_{\Omega}\phi(x, |\nabla u_i|)^{1 + \frac{\gamma}{4}} \, dx \leq C,
\end{align*}
where $C$ is a constant independent of $i$. }

 Let us first choose $\theta \in(0, \gamma)$ so that   $ \frac{1+\gamma/2}{1+\theta} -1 = \frac{\gamma}{4}$ i.e.
\begin{align}\label{eq:gamma-theta}
\left (1+\frac{\gamma}4\right )(1+\theta)= 1 + \frac{\gamma}{2}.
\end{align} 
Also assume that $i \ge i_\theta$.

By \cite[Lemma 4.6]{Kar21} we can assume $f \geq \psi$ almost everywhere in $\Omega$ and therefore using $f$ as a test function to the $(\mathcal{A}_i, \mathcal{K}_{\psi_i}^{f,\phi}(\Omega))$-obstacle problem we have
\begin{align*}
\int_{\Omega} \mathcal{A}_i(x, \nabla u_i) \cdot \nabla(f-u_i) \, dx &\ge 0 \\
\Leftrightarrow \int_{\Omega}\mathcal{A}_i(x, \nabla u_i) \cdot \nabla u_i \, dx &\le \int_{\Omega}\mathcal{A}_i(x, \nabla u_i) \cdot \nabla f \, dx.
\end{align*}
Using the structural assumptions on $\mathcal{A}_i$ in the first inequality and then continuing with Young's inequality and \adec{}, we obtain that 
\begin{align*}
c_1 \int_{\Omega} \phi_i(x,|\nabla u_i|) \, dx &\leq  c_2 \int_{\Omega} \dfrac{\phi_i(x, |\nabla u_i|)}{|\nabla u_i|} |\nabla f|\, dx \\
&\leq \dfrac{c_1}{2} \int_{\Omega}\phi_i^{\ast}\left (x, \dfrac{\phi_i(x,|\nabla u_i|)}{|\nabla u_i|}\right ) \, dx +  \int_{\Omega} \phi_i(x,\tfrac{2c_2}{c_1}|\nabla f|) \, dx \\
&\le \dfrac{c_1}{2} \int_{\Omega}\phi_i(x, |\nabla u_i|) \, dx  + C(c_1, c_2, q) \int_{\Omega} \phi_i(x,|\nabla f|) \, dx .
\end{align*}
This yields with \eqref{eq:exponent}
\begin{equation}\label{eq:phi_i-u_i}
\begin{split}
\int_{\Omega}\phi_i(x, |\nabla u_i|) \, dx \lesssim \int_{\Omega}\phi_i(x, |\nabla f|) \, dx \lesssim |\Omega| + \int_{\Omega}\phi(x,|\nabla f|)^{1+\theta} \, dx < \infty.
\end{split}
\end{equation}
Next we use the global higher integrability of $|\nabla u_i|$.
Now with \eqref{eq:exponent}, higher integrability with $\gamma/2$ (Proposition~\ref{pro:higher-integrability}), and \eqref{eq:phi_i-u_i}   we obtain
\begin{equation}\label{eq:phi_i-u_i-1b}
\begin{split}
 \int_{\Omega} \phi(x,|\nabla u_i|)^{1+\frac{\gamma}{4}} \, dx &\lesssim   \int_{\Omega} (\phi_i(x,|\nabla u_i|)^{1+\theta}  +1) ^{1+ \frac{\gamma}4} \, dx  \\
 &\lesssim   \int_{\Omega} \phi_i(x,|\nabla u_i|) ^{1+\frac{\gamma}2}  +1 \, dx \\
&\lesssim\left (\int_{\Omega} \phi_i(x,|\nabla u_i|) \, dx \right )^{1+\frac{\gamma}{2}}
+  \int_{\Omega}\phi_i(x,|\nabla f|)^{1+\frac{\gamma}{2}}\, dx \\
&\quad +\int_{\Omega}\phi_i(x,|\nabla \psi|)^{1+\frac{\gamma}{2}}\, dx  +|\Omega|  +1 \\
&\lesssim\left (|\Omega| + \int_{\Omega}\phi(x,|\nabla f|)^{1+\gamma} \, dx \right)^{1+\frac{\gamma}{2}}
+  \int_{\Omega}\phi(x,|\nabla f|)^{1+\gamma}\, dx  \\
&\quad  +  \int_{\Omega}\phi(x,|\nabla \psi|)^{1+\gamma}\, dx +|\Omega| +1 <\infty.\\
\end{split}
\end{equation}
We have proven our first claim.

\medskip
\noindent \textbf{Claim 2:} \textit{There exists a constant $C$  such that
$\|u_i\|_{W^{1,\phi^{1+\frac{\gamma}{4}}}(\Omega)} \leq C$
for all $i \ge i_\theta$.}

By \eqref{eq:phi_i-u_i-1b} and \adec{} we have $\|\nabla u_i\|_{\phi_i^{1+\frac{\gamma}{2}}}<c<\infty$.
Let us recall the Sobolev inequality in generalized Orlicz spaces \cite[Theorem 6.3.2]{HarH_book}:
\begin{align*}
\|v\|_{\eta_i} \lesssim \|\nabla v\|_{\phi_i},
\end{align*}
where $v \in W^{1, \phi_i}_0 (\Omega)$ and $\eta_i^{-1}(x,t) = t^{-1/n} \phi_i^{-1}(x,t)$, and the constant depends only on the constants on \azero{}, \aone{}, \ainc{} and \adec{}. For this we need $\phi_i$ to satisfy \adec{q} with $q<n$.

Let us quickly show that $\phi_i^{\frac{n}{n-1}}(x, t) \lesssim \eta_i(x,t) +1$ so that $\|\cdot\|_{\phi_i^{n/(n-1)}} \lesssim \|\cdot \|_{\eta_i}$. If $t \geq 1$, then from \ainc{1} of $\phi_i$ and  $\phi_i(x, \phi_i^{-1}(x, t)) \approx t$ \cite[Lemma 2.3.9]{HarH_book} (with the fact that due to \adec{}, $\phi_i(x,t)$ is always finite) we get by \ainc{1} that 
\begin{align*}
\phi_i\left (x,\eta_i^{-1}(x,t)\right )^{\frac{n}{n-1}} &= \phi_i\left (x, t^{-1/n}\phi_i^{-1}(x,t)\right )^{\frac{n}{n-1}} \lesssim t^{-\frac{1}{n}\frac{n}{n-1}}\phi_i\left (x,\phi_i^{-1}(x,t)\right )^{\frac{n}{n-1}} \\
&\approx t^{-\frac{1}{n-1}} t^{\frac{n}{n-1}} =t.
\end{align*}
Applying this to $\eta_i(x, s) = t \ge 1$ , we obtain $\phi_i(x, s)^{\frac{n}{n-1}} \lesssim \eta_i(x, s)$.
If on the other hand, $t \leq 1$, then instead of \ainc{1} we use \ainc{1/n} of $\phi_i^{-1}$, which is equivalent to \adec{n} of $\phi_i$, to have $\frac{\phi_i^{-1}(x,t)}{t^{1/n}}\leq L_{1/n} \phi^{-1}(x,1)$. Now this and \adec{} to pull $L_{1/n}$ out of $\phi$ yield
\begin{align*}
\phi_i(x,\eta_i^{-1}(x,t))^{\frac{n}{n-1}} &= \phi_i\left (x, t^{-1/n}\phi^{-1}(x,t)\right )^{\frac{n}{n-1})} \lesssim  \phi_i\left (\phi_i^{-1}(x,1)\right )^{\frac{n}{n-1}} \approx 1.
\end{align*}
 Applying this to $\eta_i(x, s) = t < 1$ , we obtain $\phi_i(x, s) \lesssim 1$.
Therefore  $\|\cdot\|_{\phi_i^{n/(n-1)}} \lesssim \|\cdot \|_{\eta_i}$ as stated.

Next we use the fact that $u_i-f \in W^{1, \phi_i}_0(\Omega)$ i.e. $u_i -f$ has zero boundary values for $\phi_i$.
Now when $1 < 1+\gamma\le \frac{n}{n-1}$, we have
\begin{align*}
\|u_i\|_{\phi_i^{1+\frac{\gamma}{2}}} \le \|u_i-f\|_{\phi_i^{1+\frac{\gamma}{2}}} + \|f\|_{\phi_i^{1+\frac{\gamma}{2}}} & \lesssim \|u_i-f\|_{\eta_i} + \|f\|_{\phi^{1+\gamma}} \lesssim \|\nabla (u_i-f)\|_{\phi_i} + \|f\|_{\phi^{1+\gamma}}\\
&\le\|\nabla u_i\|_{\phi_i} + \|\nabla f\|_{\phi^{1+\gamma}} + \|f\|_{\phi^{1+\gamma}} < \infty,
\end{align*}
where the last term is bounded by \eqref{eq:phi_i-u_i} and the latter two terms are  bounded since 
$f \in  W^{1,\phi^{1+\gamma}}(\Omega)$.  Thus we have that $\|u_i\|_{W^{1, \phi_i^{1+\frac{\gamma}{2}}}(\Omega)}$ 
is uniformly bounded. 
Since  $\left (1+\frac{\gamma}4\right )(1+\theta)= 1+\frac{\gamma}{2}$, we have
\[
\phi(x, t)^{1+\frac{\gamma}{4}} \lesssim (1+ \phi_i(x, t)^{1+ \theta} )^{1+\frac{\gamma}{4}}\lesssim 1+ \phi_i(x, t)^{1+\frac{\gamma}2}
\]
and hence  
the embedding $L^{\phi_i^{1+ \frac{\gamma}2}}(\Omega) \hookrightarrow L^{\phi^{1+\frac{\gamma}{4}}}(\Omega)$  holds. 
Thus $\|u_i\|_{L^{\phi^{1+ \frac{\gamma}{4}}}(\Omega)}$  is uniformly bounded. Hence with Claim 1 we obtain that
 $\|u_i\|_{W^{1, \phi^{1+ \frac{\gamma}{4}}}(\Omega)}$ 
is uniformly bounded.

\medskip\noindent \textbf{Claim 3:} \textit{There exists a subsequence such that $\nabla u_i \to \nabla \overline{u}$ pointwise almost everywhere in $\Omega$.}

By Claim 2 and reflexivity of $W^{1, \phi^{1+ \frac{\gamma}{4}}}$, we can choose a subsequence of $(u_i)$ (with same indexing) 
and a function $\overline{u}$ such that $u_i \rightharpoonup \overline{u}$ in $W^{1,\phi^{1+\frac{\gamma}{4}}}(\Omega)$. 
By the compact embedding   $W^{1,\phi^{1+\frac{\gamma}{4}}}(\Omega) \hookrightarrow\hookrightarrow L^{\phi^{1+\frac{\gamma}{4}}}(\Omega)$ \cite{HarHJ_pp} we obtain  $u_i \to \overline{u}$ in $L^{\phi^{1+\frac{\gamma}{4}}}(\Omega)$.
And by taking a subsequence we also have $u_i \to \overline{u}$ pointwise almost everywhere in $\Omega$. Since $u_i \geq \psi$ a.e in $\Omega$, due to pointwise convergence, this holds for $\overline{u}$, too.

Let $G\Subset G' \Subset \Omega$ and choose $\varepsilon >0$. We denote for every $i \in \N$
\begin{align*}
E_i^{\varepsilon} := \{x \in G : (\mathcal{A}_i(x, \nabla u_i)-\mathcal{A}_i(x, \nabla \overline{u})) \cdot (\nabla u_i - \nabla \overline{u}) > \varepsilon\}.
\end{align*}
Let $\eta \in C^{\infty}_0(G')$ be a cutoff function such that $\eta(x) \in [0,1]$, and $\eta(x)=1$ for all $x \in G$. We see that
\begin{align*}
\dfrac{1}{\varepsilon} \int_{G' \cap \{|u_i-\overline{u}|<\varepsilon^{2}\}}&  (\mathcal{A}_i(x, \nabla u_i)-\mathcal{A}_i(x, \nabla \overline{u})) \cdot (\nabla u_i - \nabla \overline{u}) \eta \,dx \\
&=\dfrac{1}{\varepsilon} \int_{G' \cap \{|u_i-\overline{u}|<\varepsilon^{2}\}} \mathcal{A}_i(x, \nabla u_i) \cdot (\nabla u_i - \nabla \overline{u}) \eta \, dx \\
&\quad - \dfrac{1}{\varepsilon} \int_{G' \cap \{|u_i-\overline{u}|<\varepsilon^{2}\}} (\mathcal{A}_i(x,\nabla \overline{u}) - \mathcal{A}(x,\nabla \overline{u}))\cdot (\nabla u_i - \nabla \overline{u}) \eta\,dx \\
&\quad -\dfrac{1}{\varepsilon} \int_{G' \cap \{|u_i-\overline{u}|<\varepsilon^{2}\}} \mathcal{A}(x, \nabla \overline{u}) \cdot (\nabla u_i - \nabla \overline{u}) \eta\, dx \\
&=: I_1^{i}+ I_2^{i} + I_3^{i}.
\end{align*}

To estimate the first term we set $w_i := \min\{\max\{0,\overline{u}+\varepsilon^2-u_i\},2\varepsilon^2\}$ and notice that 
$\eta w_i \in W^{1,\phi^{1+\frac{\gamma}{4}}}_0(G')\subset W^{1,\phi_i}_0(G')$.
Since $u_i$ is a solution to an $(\mathcal{A}_i, \mathcal{K}_\psi^{f,\phi_i}(\Omega))$-obstacle problem, it is a supersolution. Therefore we can test the equation with non-negative test functions such as $\eta w_i$ to get
\begin{align*}
 \int_{G' \cap \{|u_i-\overline{u}|<\varepsilon^{2}\}} \mathcal{A}_i(x, \nabla u_i) \cdot (\eta \nabla(u_i-\overline{u}))\,dx &\leq \int_{G'}\mathcal{A}_i(x,\nabla u_i)\cdot (w_i \nabla \eta) \, dx \\
&\lesssim \varepsilon^2 \int_{G'} |\mathcal{A}_i(x, \nabla u_i)| \, dx \\
& \lesssim \varepsilon^{2} \int_{G'} \dfrac{\phi_i(x,|\nabla u_i|)}{|\nabla u_i|} \, dx,
\end{align*}
where the constants depends on $\|\nabla \eta\|_\infty$, 
and thus with Young's inequality and \azero{} we have
\begin{align*}
I_1^i \lesssim \varepsilon \int_{G'} \dfrac{\phi_i(x,|\nabla u_i|)}{|\nabla u_i|} \, dx &\leq \varepsilon \int_{G'} \phi_i^{\ast}\left (\dfrac{\phi_i(x,|\nabla u_i|)}{|\nabla u_i|}\right) + \phi_i(x,1) \, dx \\
&\lesssim \varepsilon \int_{G'} \phi_i(x,|\nabla u_i|) +1 \,dx \lesssim \varepsilon,
\end{align*}
where the last inequality follows from \eqref{eq:phi_i-u_i}. Note that we do not know the sign of $I_1^i$.

Let us now turn our attention to $I^{i}_{2}$.
Let us  write here that $\zeta:= \phi^{1+ \frac{\gamma}{4}}$.
We estimate with H\"older's inequality
\begin{align*}
|I^{i}_{2}| \lesssim \dfrac{1}{\varepsilon} \|\mathcal{A}_i(x, \nabla \overline{u}) - \mathcal{A}(x,\nabla \overline{u})\|_{\zeta^{\ast}}\|\nabla u_i - \nabla \overline{u}\|_{\zeta}
\end{align*}
Next we show that
\begin{align*}
\int_{\Omega} \zeta^{\ast}(x,|\mathcal{A}_i(x,\nabla \overline{u}) - \mathcal{A}(x,\nabla \overline{u})|) \, dx \to 0
\end{align*}
when $i \to \infty$. 
By assumptions $\mathcal{A}_i(x,\nabla \overline{u}) \to \mathcal{A}(x,\nabla \overline{u})$ almost everywhere. By \ainc{1} we have
$\zeta^{\ast}(x, t) \lesssim \zeta^{\ast}(x, 1) t$ for $t \in (0, 1)$, and $\zeta^{\ast}(x, 1)$ is finite by \azero{}, and thus
$\zeta^{\ast}(x,|\mathcal{A}_i(x,\nabla \overline{u}) - \mathcal{A}(x,\nabla \overline{u})|) \to 0$ a.e.
Next we show that the integrand has an integrable majorant. For this we use \adec{} of $\zeta^{\ast}$ for triangle inequality and structural conditions of $\mathcal{A}$ and $\mathcal{A}_i$ to get
\begin{align*}
\zeta^{\ast}(x, |\mathcal{A}_i(x,\nabla \overline{u}) - \mathcal{A}(x,\nabla \overline{u})|)| &\lesssim \zeta^{\ast}(x,|\mathcal{A}_i(x, \nabla \overline{u})|) + \zeta^{\ast}(x,| \mathcal{A}(x, \nabla \overline{u})|) \\
&\lesssim \zeta^{\ast}\left (x, \dfrac{\phi_i(x,|\nabla \overline{u}|)}{|\nabla \overline{u}|}\right ) + \zeta^{\ast}\left (x, \dfrac{\phi(x, |\nabla \overline{u}|)}{|\nabla \overline{u}|}\right ).
\end{align*}
Let us first estimate the first term. If $|\nabla \overline{u}| \leq 1$ we can use \ainc{1} of $\phi_i$ and \adec{} with \azero{} of $\phi_i$ and $\zeta^{\ast}$ to estimate
\begin{align*}
\zeta^{\ast}\left (x, \dfrac{\phi_i(x, |\nabla \overline{u}|)}{|\nabla u|}\right ) \lesssim \zeta^{\ast}\left (x, \phi_i(x, 1)\right ) \lesssim \zeta^{\ast}(x, 1) \lesssim 1.
\end{align*}
Assume then that $|\nabla \overline{u}| >1$. 
Then  by \eqref{eq:exponent} we choose $i$ so large that with $\phi_i \le 1 + \phi^{1+\frac{\gamma}{4}}$ and with \adec{} we obtain
\[
\begin{split}
\zeta^{\ast}\left (x, \dfrac{\phi_i(x, |\nabla \overline{u}|)}{|\nabla \overline{u}|}\right )
&\le \zeta^{\ast}\left (x, \dfrac{\zeta(x, |\nabla \overline{u}|) +1}{|\nabla \overline{u}|}\right )
\le \zeta^{\ast}\left (x, \dfrac{\zeta(x, |\nabla \overline{u}|)}{|\nabla \overline{u}|}+1\right )\\
&\lesssim  \zeta^{\ast} \left (x, \dfrac{\zeta(x, |\nabla \overline{u}|)}{|\nabla \overline{u}|}\right )
+\zeta^{\ast} (x, 1) \lesssim \zeta(x, |\nabla \overline{u}|) + \zeta^{\ast} (x, 1)\\
&= \phi(x, |\nabla \overline{u}|)^{1+ \frac{\gamma}{4}} + \zeta^{\ast} (x, 1),
\end{split}
\]
where the last term is uniformly bounded since $\zeta^{\ast}$ satisfies \azero{}.

Let us then estimate the second term. If $|\nabla \overline{u}| \leq 1$ we act as in the previous case and obtain 
\[
\zeta^{\ast}\left (\dfrac{\phi(x, |\nabla \overline{u}|)}{|\nabla \overline{u}|}\right ) \lesssim 1.
\]
If $|\nabla \overline{u}| >1$, then we estimate
\[
\zeta^{\ast}\left (x, \dfrac{\phi(x, |\nabla \overline{u}|)}{|\nabla \overline{u}|}\right )
\le \zeta^{\ast}\left (x, \dfrac{\zeta(x, |\nabla \overline{u}|)+1}{|\nabla \overline{u}|}\right )
\le \zeta^{\ast}\left (x, \dfrac{\zeta(x, |\nabla \overline{u}|)}{|\nabla \overline{u}|}+1\right )
\]
and rest is as in the previous case.    

Combining both cases we end up with 
\begin{align*}
\zeta^{\ast}(x, |\mathcal{A}_i(x,\nabla \overline{u}) - \mathcal{A}(x,\nabla \overline{u})|) \lesssim 1 + \phi(x, |\nabla \overline{u}|)^{1+\frac{\gamma}4} \in L^{1}(\Omega).
\end{align*}
Now the convergence follows from Lebesgue's dominated convergence theorem so $|I_2^{i}| \to 0$.

To estimate the last term we begin with the observation
\[
\phi^*\Big( x, |\mathcal{A}(x, \nabla \overline{u})| \Big) \le \phi^*\bigg( x,  c_2 \frac{\phi(x, |\nabla \overline{u}|)}{|\nabla \overline{u}|} \bigg ) \lesssim \phi(x, |\nabla \overline{u}|),
\]
and hence $|\mathcal{A}(x, \nabla \overline{u})| \in L^{\phi^{\ast}}(\Omega)$.
Since
\begin{align*}
I_3^{i} = -\dfrac{1}{\varepsilon} \int_{G' \cap \{|u_i-\overline{u}|<\varepsilon^{2}\}} \mathcal{A}(x, \nabla \overline{u}) \cdot (\nabla u_i - \nabla \overline{u}) \eta\, dx
\end{align*}
and  $\frac{\partial}{\partial x_j} u_i \rightharpoonup \frac{\partial}{\partial x_j} \overline{u}$  in $ L^{\phi^{1+\frac{\gamma}{4}}}(\Omega) \subset L^{\phi}(\Omega)$ for every $j=1, \ldots, n$, we see that $|I_3^{i}| < \varepsilon$ for $i$ large enough.

All in all
\begin{align*}
|E_i^{\varepsilon}| &\leq |E^{\varepsilon}_i \cap \{|u_i-\overline{u}|<\varepsilon^{2}\}| + |E^{\varepsilon}_i \cap \{|u_i-\overline{u}|\geq \varepsilon^{2}\}| \\
&\leq \dfrac{1}{\varepsilon} \int_{G'\cap \{|u_i-\overline{u}|<\varepsilon^2\}} (\mathcal{A}_i(x,\nabla u_i)-\mathcal{A}_i(x, \nabla \overline{u}))\cdot(\nabla u_i - \nabla \overline{u})\, dx \\
&\quad + |E_i^{\varepsilon}\cap \{|u_i-\overline{u}|\geq \varepsilon^2\}| \\
&= I_1^{i}+ I_2^{i} + I_3^{i}  + |E_i^{\varepsilon}\cap \{|u_i-\overline{u}|\geq \varepsilon^2\}| \\
&\le I_1^{i}+ |I_2^{i}| + |I_3^{i}|+ |E_i^{\varepsilon}\cap \{|u_i-\overline{u}|\geq \varepsilon^2\}| \\
&\lesssim  \varepsilon+ |\{|u_i-\overline{u}| \geq \varepsilon^{2}\}| \lesssim \varepsilon
\end{align*}
since $u_i \to \overline{u}$ almost everywhere and therefore also in measure. Thus
\begin{align}\label{eq:I_i(x)}
\lim_{i\to \infty} (\mathcal{A}_i(x, \nabla u_i)-\mathcal{A}_i(x, \nabla \overline{u})) \cdot (\nabla u_i - \nabla \overline{u}) = 0
\end{align}
for almost everywhere in $G$.

To simplify notation, let us denote 
$I_i(x) := (\mathcal{A}_i(x, \nabla u_i)-\mathcal{A}_i(x, \nabla \overline{u})) \cdot (\nabla u_i - \nabla \overline{u})$. 
Now for almost every $x_0 \in G$ we have
\begin{itemize}
\item $I_i(x_0) \to 0 $ as $i \to \infty$; 
\item $|\nabla \overline{u}(x_0)| < \infty$;
\item Structural conditions hold at $x_0$;
\item $\mathcal{A}_i(x_0,\xi) \to \mathcal{A}(x_0,\xi)$.
\end{itemize}
Now it cannot be that $|\nabla u_i(x_0)| \to \infty$ since then
\begin{align*}
c_1 \phi_i(x_0, |\nabla u_i(x_0)|) &\leq \mathcal{A}_i(x_0,\nabla u_i(x_0))\cdot \nabla u_i(x_0)\\
&= I_i(x_0) + \mathcal{A}_i(x_0, \nabla \overline{u}(x_0))\cdot \big(\nabla u_i(x_0) - \nabla \overline{u}(x_0)\big) \\
&\quad + \mathcal{A}_i(x_0,\nabla u_i(x_0))\cdot \nabla \overline{u}(x_0) \\
&\leq I_i(x_0) + c_2 \dfrac{\phi_i(x_0,|\nabla \overline{u}(x_0)|)}{|\nabla \overline{u}(x_0)|}\big(|\nabla u_i(x_0)| + |\nabla \overline{u}(x_0)| \big) \\
&\quad + c_2 \dfrac{\phi_i(x_0,|\nabla u_i(x_0)|)}{|\nabla u_i(x_0)|} |\nabla \overline{u}(x_0)|.
\end{align*}
But now the left hand side would grow like the $t\mapsto \phi_i(x,t)$ and the right hand side like 
$t \mapsto  t+\frac{\phi_i(x,t)}{t}$ which is a contradiction with \ainc{}. By monotonicity we see that if $\lim_{i \to \infty} \nabla u_i(x_0) =: \xi \not = \nabla \overline{u}(x_0)$, then
\begin{align*}
0 = \lim_{i \to \infty} I_i(x_0) = (\mathcal{A}(x_0,\xi) - \mathcal{A}(x_0, \nabla \overline{u}(x_0)) \cdot (\xi - \nabla \overline{u}(x_0)) >0,
\end{align*}
which is again a contradiction, so it has to be that $\nabla u_i \to \nabla \overline{u}$ almost everywhere and Claim 3 is proved.

\medskip\noindent \textbf{Claim 4:} \textit{$u_i \to \overline{u}$ in $W^{1,\phi^{1+\delta}}(\Omega)$ for every $\delta\in [0,\frac{\gamma}{4})$.}

Since $u_i \to \overline{u}$ strongly in $L^{\phi^{1+\frac{\gamma}{2}}}(\Omega)$, weakly in $W^{1,\phi^{1+\frac{\gamma}{4}}}(\Omega)$ and $|\nabla u_i| \to |\nabla \overline{u}|$ pointwise, we actually have $u_i \to \overline{u}$ strongly in $W^{1,\phi^{1+\delta}}(\Omega)$ for every $\delta<\frac{\gamma}{4}$. Indeed, defining $v_i:=\phi(x,|\overline{u}-u_i|)^{1+\delta} + \phi(x,|\nabla \overline{u}-\nabla u_i|)^{1+\delta}$ we see that $v_i \to 0$ almost everywhere in $\Omega$ and $v_i$ is equiintegrable in $L^{1}(\Omega)$, since the sequence is bounded in $L^{r}(\Omega)$, where $r=\frac{1+\gamma/4}{1+\delta}>1$. Now Vitali's convergence theorem \cite[p. 374]{Vitali} yields the claim.

\medskip\noindent \textbf{Claim 5:} \textit{$\overline{u}$ has $f$ as its boundary values in the Sobolev sense.} 

Since $f \in W^{1,\phi^{1+\frac{\gamma}{4}}}(\Omega)$, $u_i-f \in W^{1,\phi_i}_0(\Omega)$ and $(u_i)$ is bounded in $W^{1,\phi^{1+\frac{\gamma}{4}}}(\Omega)$ (Claim 2), it follows from \adec{}, largeness of $i$ and Claim 1 that
\begin{align*}
\int_{\Omega}\phi_i(x, |\nabla (u_i-f)|) \, dx \lesssim |\Omega| + \int_{\Omega} \phi(x,|\nabla u_i|)^{1+\frac{\gamma}{2}} \,dx + \int_{\Omega} \phi(x,|\nabla f|)^{1+\frac{\gamma}{2}} \, dx \leq M.
\end{align*}
 Lemma \ref{lemma:LM} shows that $\overline{u}-f \in W^{1,\phi}_0(\Omega)$ if we know that $\phi_i(x,t_i) \to \phi(x,t)$ whenever $t_i \to t$. 
To obtain this, first note that
\begin{align*}
|\phi_i(x,t_i) - \phi_i(x,t)| < \frac{\varepsilon}{2}
\end{align*}
when $i > i_1 \in \N$ since $(\phi_i)$ is uniformly equicontinuous. Additionally
\begin{align*}
|\phi_i(x,t) - \phi(x,t)| < \frac{\varepsilon}{2}
\end{align*}
when $i > i_2 \in \N$ since $\phi_i(x,t) \to \phi(x,t)$ by assumption in $\Omega$ for a fixed $t$. Therefore
\begin{align*}
|\phi_i(x,t_i) - \phi(x,t) | < |\phi_i(x,t_i) - \phi_i(x,t)| + |\phi_i(x,t) - \phi(x,t)|  < \varepsilon
\end{align*}
when $i > \max\{i_1, i_2\}$.

\medskip\noindent \textbf{Claim 6:} \textit{We have $\overline{u}=u$.} 

We can write $\overline{u}-u = (\overline{u}-f) + (f-u) \in W^{1,\phi}_0(\Omega)$. 
Note that $\overline{u} \in \mathcal{K}_{\psi}^{f,\phi}(\Omega)$, since $\overline{u} -f \in W^{1, \phi}_0(\Omega)$ by Claim 5, and  the pointwise convergence $u_i \to \overline{u}$ with $u_i \ge \psi$ yield $\overline{u} \ge \psi$.  So we can use 
$\overline{u}$ as a test function for the $(\mathcal{A},\mathcal{K}_{\psi}^{f,\phi}(\Omega))$-obstacle problem and obtain
\begin{align}\label{equ:ey1}
\int_{\Omega}\mathcal{A}(x, \nabla u) \cdot (\nabla \overline{u} - \nabla u) \, dx \geq 0.
\end{align}
To have a similar inequality
\begin{align}\label{equ:ey3}
\int_{\Omega}\mathcal{A}(x, \nabla \overline{u}) \cdot (\nabla u-\nabla \overline{u}) \, dx \geq 0,
\end{align}
we need some additional effort.

For any $w \in \mathcal{K}_{\psi}^{f,\phi^{1+\gamma}}(\Omega)$ we can assume that $w \in \mathcal{K}_{\psi}^{f,\phi_i}(\Omega)$ for $i$ large enough. Using $w$ as a test function for the obstacle problem we get
\begin{align*}
\int_{\Omega} \mathcal{A}_i(x, \nabla u_i) \cdot (\nabla w- \nabla u_i) \, dx \geq 0.
\end{align*}
We show by means of Vitali's convergence theorem \cite[p. 374]{Vitali} that letting $i\to \infty$ in the previous inequality yields
\begin{align*}
\int_{\Omega} \mathcal{A}(x, \nabla \overline{u}) \cdot (\nabla w- \nabla \overline{u}) \, dx  \geq 0.
\end{align*}
For this we need to show that integrands converge pointwise and that the sequence of integrals is absolutely continuous.

The pointwise convergence follows from
\begin{align*}
\mathcal{A}_i(x, \nabla u_i) \cdot (\nabla w - \nabla u_i) &=\mathcal{A}_i(x, \nabla u_i) \cdot (\nabla w - \nabla u_i) -\mathcal{A}_i(x,\nabla u_i)\cdot(\nabla \overline{u} - \nabla u_i) \nonumber\\
&\quad + \mathcal{A}_i(x,\nabla u_i)\cdot(\nabla \overline{u} - \nabla u_i)  - \mathcal{A}_i(x,\nabla \overline{u})\cdot(\nabla \overline{u} - \nabla u_i) \nonumber\\
&\quad + \mathcal{A}_i(x,\nabla \overline{u})\cdot(\nabla \overline{u} - \nabla u_i)  - \mathcal{A}(x,\nabla \overline{u})\cdot(\nabla \overline{u} - \nabla u_i)\nonumber \\
&\quad + \mathcal{A}(x,\nabla \overline{u})\cdot(\nabla \overline{u} - \nabla u_i)  - \mathcal{A}(x, \nabla \overline{u})\cdot(\nabla w - \nabla u_i) \nonumber\\
&\quad + \mathcal{A}(x, \nabla \overline{u})\cdot(\nabla w - \nabla u_i)\nonumber \\
&=\mathcal{A}_i(x,\nabla u_i) \cdot(\nabla w - \nabla \overline{u}) - I_i(x)  \\
&\quad +(\mathcal{A}_i(x,\nabla \overline{u})-\mathcal{A}(x,\nabla \overline{u}))\cdot(\nabla \overline{u} - \nabla u_i) \\
&\quad + \mathcal{A}(x,\nabla \overline{u})\cdot(\nabla \overline{u} - \nabla w) \nonumber\\
&\quad + \mathcal{A}(x,\nabla \overline{u})\cdot (\nabla w - \nabla u_i)\nonumber \\
&= \underbrace{\mathcal{A}_i(x,\nabla u_i) - \mathcal{A}(x,\nabla \overline{u}))\cdot(\nabla w - \nabla \overline{u})}_{(a)} - I_i(x)\nonumber \\
&\quad + \underbrace{(\mathcal{A}_i(x,\nabla \overline{u})-\mathcal{A}(x,\nabla \overline{u}))\cdot(\nabla \overline{u} - \nabla u_i)}_{(b)} \\
&\quad + \underbrace{\mathcal{A}(x,\nabla \overline{u})\cdot (\nabla w - \nabla u_i)}_{(c)}.\nonumber
\end{align*}

Let us next estimate the above terms:
\begin{itemize}
\item[(a)] First, we estimate
\begin{flalign*}
&&|\mathcal{A}_i(x, \nabla u_i) - \mathcal{A}(x,\nabla \overline{u})| \leq |\mathcal{A}_i(x,\nabla u_i) - \mathcal{A}_i(x, \nabla \overline{u})| + |\mathcal{A}_i(x, \nabla \overline{u}) - \mathcal{A}(x, \nabla \overline{u})|.
\end{flalign*}
Secondly, note that $|\mathcal{A}_i(x,\nabla u_i) - \mathcal{A}_i(x, \nabla \overline{u})|$ converges to zero by uniform equicontinuity of 
$(\mathcal{A}_i)$ together with pointwise convergence $\nabla u_i \to \nabla \overline{u}$ (Claim 3)
and $ |\mathcal{A}_i(x, \nabla \overline{u}) - \mathcal{A}(x, \nabla \overline{u})|$ converges to zero due to the pointwise convergence of $(\mathcal{A}_i)$. Therefore $(a)$ converges to 0.
\item[$I_i(x)$]We showed in the end of Claim 3 that $-I(x)$ converges to $0$ almost everywhere in  $G$. Since $G \Subset \Omega$ was arbitray, the pointwise convergence holds for almost every $x \in \Omega$.
\item[(b)] Since $\mathcal{A}_i(x, \nabla \overline{u}) \to \mathcal{A}(x, \nabla \overline{u})$ and $\nabla u_i \to \nabla \overline{u}$ almost everywhere by Claim 3, the term converges to 0.
\item[(c)] Since $\nabla u_i \to \nabla \overline{u}$ almost everywhere (Claim 3), this last term converges to $\mathcal{A}(x, \nabla \overline{u}) \cdot (\nabla w - \nabla \overline{u})$.
\end{itemize}

Next we show the sequence of integrals is absolutely continuous. 
Let $E \subset \Omega$. 
 Let us assume that $\theta < \frac{\gamma}{4}$, and also assume that $i \ge i_\theta$.
From the structural assumptions, uniform convergence, Young's inequality, and Hölder's inequality we get
\begin{align*}
&\Big |\int_{E}\mathcal{A}_i(x, \nabla u_i) \cdot (\nabla w - \nabla u_i)\, dx\Big | \lesssim \int_{\Omega }\phi_i(x, |\nabla u_i|) + \dfrac{\phi_i(x, |\nabla u_i|)}{|\nabla u_i|}|\nabla w|\, dx  \\
&\quad\lesssim \int_{E} \phi_i(x, |\nabla u_i|) + \phi_i^{\ast}\left (\dfrac{\phi_i(x, |\nabla u_i|)}{|\nabla u_i|}\right ) + \phi_i(x, |\nabla w|)\, dx\\
&\quad\lesssim \int_{E} \phi_i(x,|\nabla u_i|)  +\phi_i(x, |\nabla w|) \,dx\\
&\quad\lesssim \int_{E} \phi(x,|\nabla u_i|)^{1+ \theta} + \phi(x, |\nabla w|)^{1+ \theta} +2 \,dx\\
&\quad\lesssim |E|^{\frac1{\lambda'}} \Big(\int_{E} \phi(x,|\nabla u_i|)^{1+\frac{\gamma}{4}} \, dx\Big)^{\frac1{\lambda}}
+ |E|^{\frac1{\lambda'}} \Big(\int_{E} \phi(x,|\nabla w|)^{1+\frac{\gamma}{4}} \, dx\Big)^{\frac1{\lambda}} + |E|,
\end{align*}
where $\lambda>1$ is chosen so that $\lambda (1+ \theta) = 1+ \frac{\gamma}{4}$.
The first integral is uniformly bounded by Claim 1, the second integral is bounded since $w \in \mathcal{K}_{\psi}^{f,\phi^{1+\gamma}}(\Omega)$.
Hence $$\int_{E}\mathcal{A}_i(x, \nabla u_i) \cdot (\nabla w - \nabla u_i)\, dx \to 0$$ uniformly as $|E| \to 0$.
Thus we can use Vitali's convergence thereom to conclude that 
\begin{align}\label{equ:ey4}
\int_{\Omega} \mathcal{A}(x, \nabla \overline{u}) \cdot (\nabla w-\nabla \overline{u}) \, dx = \lim_{i\to \infty}\int_{\Omega}\mathcal{A}_i(x, \nabla u_i) \cdot (\nabla w - \nabla u_i) \, dx \geq 0
\end{align}
which is nearly in the form of \eqref{equ:ey3}, but we want to replace $w$ by $u$.

Since $u - f \in W^{1,\phi}_0(\Omega)$, by definition, there exists a sequence $(\eta_i) \subset C^{\infty}_0(\Omega)$ such that $\eta_i \to u-f$ in $W^{1,\phi}(\Omega)$.
As  $\eta$ has compact support and we can always assume that $f \geq \psi$, we deduce that $\max\{\eta_i, \psi-f\} \in W^{1,\phi^{1+\gamma}}_0(\Omega)$. 
This allows us to define functions $v_i := \max\{\eta_i, \psi-f\}+f$ and immediately conclude that $v_i \in \mathcal{K}_{\psi}^{f,\phi^{1+\gamma}}(\Omega)$.

We note that $u\geq \psi$ a.e. in $\Omega$ and so
\begin{align*}
v_i-u = \max\{\eta_i-(u-f), \psi-u\} \to 0
\end{align*}
in $W^{1,\phi}(\Omega)$. Therefore the gradients converge also in the weak sense and thus letting $w = v_i$ in \eqref{equ:ey4} we get
\begin{align*}
\int_{\Omega }\mathcal{A}(x, \nabla \overline{u}) \cdot (\nabla u - \nabla \overline{u}) \, dx = \lim_{i \to \infty} \int_{\Omega }\mathcal{A}(x, \nabla \overline{u}) \cdot (\nabla v_i - \nabla u) \, dx  \geq 0
\end{align*}
which is \eqref{equ:ey3}.

Finally, by the monotonicity of $\mathcal{A}$ and noting inequalities \eqref{equ:ey1} and \eqref{equ:ey3}, we see that
\begin{align*}
0 \leq \int_{\Omega}( \mathcal{A}(x, \nabla \overline{u}) - \mathcal{A}(x, \nabla u)) \cdot (\nabla \overline{u} - \nabla u) \, dx \leq 0,
\end{align*}
so by the structural assumption (4) it has to be that $\nabla \overline{u} = \nabla u$  almost everywhere. Function $u$ has boundary values $f$, and so does function $\overline{u}$ by Claim 5. Thus by Poincar\'e inequality we obtain
\[
\|u-\overline{u}\|_\phi \lesssim \|\nabla (u-\overline{u})\|_\phi=0 
\]
and hence $\overline{u} = u$ almost everywhere.

\medskip

We have proved that the sequence $(u_i)$ of 
solutions to \eqref{eq:equation} with operator $\mathcal{A}_i$ has a subsequence which converges to the solution $u$  to \eqref{eq:equation} with operator $\mathcal{A}$ in 
$W^{1, \phi^{1+\gamma}}(\Omega)$. 
\end{proof}

From the proof it can be seen that $\delta$ can be chosen as any positive number satisfying $\delta < \gamma/4$, see Claim~4.

Next we show that the subsequence converges also in some H\"older space.

\begin{proof}[Proof of Corollary \ref{cor:holder}]
Let us first recall that the solution $u_i$ to an $(\mathcal{A}_i,\mathcal{K}_{\psi}^{f,\phi_i}(\Omega))$-obstacle problem is locally bounded (see \cite[Proposition 5.3]{Kar21}). This bound and Claim 2 yield 
\begin{align*}
\esssup_{Q_R} u_i &\le k_0 + c R^{- \frac{q}{s p}} \bigg( \int_{Q_{2R}} \phi_i(x, \max\{u_i-k_0, 0\} )\, dx \bigg)^{\frac1p} \\
&\leq k_0 + c R^{- \frac{q}{s p}} \bigg( \int_{Q_{2R}} \phi_i(x, u_i )\, dx \bigg)^{\frac1p} \\
&\leq k_0 + cR^{-\frac{q}{s p}} \bigg( \int_{Q_{2R}} \phi(x, u_i )^{1+\gamma/2}+1 \, dx \bigg)^{\frac1p} \leq C,
\end{align*}
where $s := \frac{p}{nq-n(q-p)}$, $k_0 \geq \sup_{x \in Q_{2R}} \psi(x)$ and $Q_R$ is a cube with side length $R$, $Q_{6 R} \subset \Omega$ and $R \le R_0$. The second inequality follows since if $\sup_{x \in Q_{2R}} \psi(x) <0$, we choose $k_0=0$ and if $k_0=\sup_{x \in Q_{2R}} \psi(x)\geq 0$ the inequality is trivial.

Since $u_i$ is a solution to an obstacle problem, it is a supersolution and therefore a quasisuperminimizer. Therefore it follows that $-u$ is a quasisubminimizer and we obtain similar estimate for 
$\esssup_{Q_R} -u_i$  \cite[Theorem 4.7]{HarHT17}. Thus we have
\[
\|u_i\|_{L^{\infty}(Q_R)} \le C.
\]

Let $D \Subset \Omega$. Thus we can cover $D$ by finitely many cubes satisfying the above condition and obtain that each $u_i$ is bounded in $D$ and the $L^\infty(D)$-norm is bounded by  constant $c$ independent of $i$. 
By adding the constant $c$ 
to  solutions we get  non-negative solutions and hence also non-negative quasiminimizers.
Then by \cite[Theorem 2.6]{KarL21} the local  non-negative  solutions to $(\mathcal{A}_i,\mathcal{K}_{\psi}^{f,\phi_i}(\Omega))$-obstacle problems are locally H\"older continuous and the constants 
depend only on  constants in $n$, \ainc{p}, \adec{q},  \azero{} and  \aone{}, and $L^\infty_{\loc}$-norm of the solution.  
By assumptions, constants in  \ainc{p}, \adec{q}, \azero{} and  \aone{} are independent of $i$.

Combining above results we see that the sequence $(u_i+c)$ is equicontinuous. 
This yields that the sequence $(u_i+c)$ 
is bounded in $C^{0,\kappa}(D)$. If $\alpha < \kappa$, there is a compact embedding 
$C^{0,\kappa}(D) \hookrightarrow C^{0,\alpha}(D)$ so $u_i \to u$ in  $C_\loc^{0,\alpha}(\Omega)$.
\end{proof}

%%%%%%%%%%%%%%%%%%%%%%%%%%%%%%%%%%%%%%%%%%%%%%%%%%%%%%%%%%%%%%%%%%%
%%%%%%%%%%%%%%%%%%%%%%%%%%%%%%%%%%%%%%%%%%%%%%%%%%%%%%%%%%%%%%%%%%%
%%%%%%%%%%%%%%%%%%%%%%%%%%%%%%%%%%%%%%%%%%%%%%%%%%%%%%%%%%%%%%%%%%%

\bigskip

\noindent\small{
\textsc{P. Harjulehto}}\\
\small{Department of Mathematics and Statistics,
FI-00014 University of Helsinki, Finland}\\
\footnotesize{\texttt{petteri.harjulehto@helsinki.fi}}\\

\noindent\small{
\textsc{A. Karppinen}}\\
\small{Institute of Applied Mathematics and Mechanics, University of Warsaw, \\ ul. Banacha 2, 02-097, Warsaw, Poland}\\
\footnotesize{\texttt{a.karppinen@uw.edu.pl}}\\
\end{document}